\documentclass[12pt]{article}
\nonstopmode
\RequirePackage[colorlinks,citecolor=blue,urlcolor=blue,linkcolor=blue]{hyperref}
\hypersetup{
colorlinks = true,
citecolor=blue,
urlcolor=blue,
linkcolor=blue,
pdfauthor = {Daniel Hackmann, Alexey Kuznetsov},
pdfkeywords = {Asian option, meromorphic process, hyper-exponential process, exponential functional,  Mellin transform, gamma function },
pdftitle = {Asian options and meromorphic Levy processes},
pdfpagemode = UseNone
}
\usepackage{graphicx,xspace,colortbl}
\usepackage{amsmath,amsthm,amsfonts}
\usepackage{color}
\usepackage{fancybox}
\usepackage{epsfig}
\usepackage{subfig}
\usepackage{pdfsync}
    \def\qed{\hfill$\sqcap\kern-8.0pt\hbox{$\sqcup$}$\\}
    \def\beq{\begin{eqnarray}}
    \def\eeq{\end{eqnarray}}
    \def\beqq{\begin{eqnarray*}}
    \def\eeqq{\end{eqnarray*}}

    \def\re{\textnormal {Re}}
    \def\im{\textnormal {Im}}
    \def\p{{\mathbb P}}
    \def\e{{\mathbb E}}
    \def\r{{\mathbb R}}
    \def\c{{\mathbb C}}
    	
    \def\d{{\textnormal d}}
    \def\i{{\textnormal i}}

    \def\ind{{\mathbb I}}

    \def\mm{{\mathcal M}}

    \def\ee{{\textnormal e}}

\newtheorem{theorem}{Theorem}
\newtheorem{lemma}{Lemma}
\newtheorem{proposition}{Proposition}

\theoremstyle{definition}

\newtheorem{remark}{Remark}

\title{Asian options and meromorphic L\'evy processes}
\author{
D. Hackmann and A. Kuznetsov
\thanks{{Research supported by the
Natural Sciences and Engineering Research Council of Canada.}}  \\ \\
Dept. of Mathematics and Statistics\\  York University \\
4700 Keele Street 
\\Toronto, ON \\ M3J 1P3,  Canada 
 }

\date{}

\begin{document}

\maketitle

\begin{abstract}
\bigskip
One method to compute the price of an arithmetic Asian option in a L\'evy driven model is based on the exponential functional of the underlying L\'evy process: If we know the distribution 
of the exponential functional, we can calculate the price of the Asian option via the inverse Laplace transform. In this paper we consider pricing Asian options in a model driven by a general meromorphic L\'evy process. We prove that the exponential functional is equal in distribution
to an infinite product of indepedent beta random variables, and its Mellin transform can be expressed as an infinite product of gamma functions. We show that these results lead to an efficient algorithm for computing 
 the price of the Asian option via the inverse Mellin-Laplace transform, and we compare this method with some other techniques. 
\end{abstract}

{\vskip 0.5cm}
{\vskip 0.5cm}

\newpage


\section{Introduction}


One of the popular models for stock price dynamics in mathematical finance 
is based on a geometric L\'evy process $S_t=S_0e^{X_t}$, where $X$ is a L\'evy process started from zero and $S_0$ is the initial stock price. 
In this model the interest rate $r$ is typically constant and it is assumed that the measure $\p$ is risk-neutral, which means that the process $S_t e^{-rt}$ is a martingale under measure $\p$.
We are interested in calculating the price of a fixed strike Asian option, which in this model would be given by
\beq\label{Asian_option}
C(S_0,K,T):=e^{-r T} \e\left[ \left(S_0\int_{0}^T e^{X_u} \d u - K\right)^+\right]. 
\eeq

In order to calculate $C(S_0,K,T)$ one could use either a Monte-Carlo simulation approach, or a
semi-analytical technique. We would like to emphasize two of these techniques, both of which have important connections with other branches of probability. The first technique is based on the observation that, 
while the process $Z_t:=\int_0^t \exp(X_u) \d u$ is not Markovian, the process 
\beqq
\tilde Z_t:=(x+Z_t)e^{-X_t}
\eeqq
does satisfy Markov property. The process $\tilde Z_t$ is known as a {\it generalized Ornstein-Uhlenbeck} process, and it is an important and well-studied object, see 
\cite{LS09} and the references therein. Since $\tilde Z$ is a Markov process, after a change of measure we can rewrite the expectation in \eqref{Asian_option} so that it involves only $\tilde Z_T$, and 
this can be computed by solving the backward Kolmogorov equation, which in this case takes the form of a partial integro-differential equation (one dimension for time and one for space variable). This approach was developed in a more general setting by Vecer and Xu \cite{Vecer_Xu}, and it was skillfully implemented by Bayraktar and Xing 
\cite{Bayraktar} for jump-diffusion processes. 

The second semi-analytical technique is based on the 
{\it exponential functional} of the L\'evy process $X$, which is defined for $q\ge 0$ as 
\beq\label{def_Iq}
I_q:=\int\limits_0^{\ee(q)} e^{X_u} \d u.
\eeq
Here $\ee(q)$ denotes an exponential random variable (independent of $X$) with mean $1/q$  if $q>0$, and 
$\ee(q)\equiv +\infty$ if $q=0$. In the latter case we need to impose an additional condition 
$\e[X_1]<0$ in order to ensure that the exponential functional $I_q$ is well-defined (see \cite[Theorem 1]{BYS}).

Exponential functionals are very important and useful objects, not only in mathematical finance but also 
in many other areas of probability theory. They play a role in 
such fields as self-similar Markov processes, 
random processes in random environment, fragmentation processes and branching processes. They are also connected with the generalized Ornstein-Uhlenbeck processes (their distribution appears as a stationary measure of the process). See
\cite{BYS} for an overview of this topic. The connection with Asian options is the following: The distribution of the exponential functional can be viewed as the Laplace transform in $t$-variable of the distribution of $Z_t$. By inverting this Laplace transform 
one can obtain the distribution of $Z_t$ and, therefore, the price of the Asian option. This approach was use in \cite{CaiKou2010} for L\'evy processes with hyper-exponential jumps (see also
\cite{Patie2009407} for general results on L\'evy processes with one-sided jumps), and this will be the approach that we will follow in the present paper. Our goal is to identify the distribution of the exponential functional for a large family of underlying L\'evy processes (having such desirable properties, as jumps of infinite activity or infinite variation) and to apply these results to pricing Asian options.

Let us say a few words on the current state of knowledge of the distribution of the exponential functional. 
This distribution is known explicitly for processes with hyper-exponential jumps \cite{CaiKou2010}, and for a slightly more general class of processes with jumps of rational transform \cite{Kuznetsov2012654} (the processes in both of these families can have jumps of finite activity only). The distribution of the exponential functional is also known for a family of hypergeometric processes, which was introduced recently in \cite{KuPa2010_prep}. This family includes processes with jumps of infinite activity or infinite variation, however in this case the distribution of $I_q$ is only known for a single value of $q>0$, which depends on the parameters of the underlying process $X$. To the best of our knowledge, there are no known examples of 
L\'evy processes having double-sided jumps of infinite activity or infinite variation, 
for which the distribution of $I_q$ can be identified explicitly for all $q>0$.

In this paper we fill this gap and we find the distribution of the exponential functional $I_q$
for a general meromorphic process. 
Meromorphic processes (which were introduced in \cite{KuzKyPa2011}) are defined via the density of the L\'evy measure
\beq\label{def_pi}
\pi(x)=\ind_{\{x>0\}} \sum\limits_{n\ge 1} a_n \rho_n e^{-\rho_n x}+ \ind_{\{x<0\}} \sum\limits_{n\ge 1} \hat a_n \hat \rho_n e^{\hat \rho_n x}.
\eeq 
All the coefficients  $a_n$, $\hat a_n$, $\rho_n$ and $\hat \rho_n$ are strictly positive and
the sequences $\{\rho_n\}_{n \ge 1}$ and $\{\hat \rho_n\}_{n\ge 1}$  must be strictly increasing and unbounded.
There are several families of meromorphic processes, for which the Laplace exponent $\psi(z):=\ln \e[\exp(zX_1)]$ can be computed explicitly, such as 
 beta, theta and hypergeometric processes, see \cite{Kuz2010a,Kuz2010b,KuKyPaSc2011, KuPa2010_prep}. 
 The family of meromorphic processes is quite large, in particular it is dense in the class of processes with completely-monotone L\'evy measure (see \cite{KuzKyPa2011}). At the same time, this family is a natural generalization of L\'evy processes
 with hyper-exponential jumps (we will call them  {\it hyper-exponential processes} from now on), for which the density of the L\'evy measure is also given by \eqref{def_pi}, with the two infinite series replaced by finite sums. 
While hyper-exponential processes can have only finite activity compound Poisson jumps, meromorphic processes
 can exhibit a wide range of path behavior, including jumps of infinite activity or infinite variation, which makes them good 
candidates for modeling purposes in mathematical finance. Meromorphic processes can be seen as analogues of the well-known models (VG, CGMY, KoBoL, NIG), but with richer analytical structure allowing for development of efficient 
semi-analytical numerical algorithms. See \cite{Madan2012,Castilla20122466} for some recent applications
of these processes in
mathematical finance.

The paper is organized as follows. In section \ref{section_inf_products} we study infinite products of independent beta random variables. We use these results in section \ref{section_exp_func_meromorphic} to
identify the distribution of the exponential functional of a general meromorphic process. In section \ref{section_numerics} we discuss numerical issues, such as approximating the Mellin transform of the exponential functional and computing the price of an Asian option..


\section{An infinite product of beta random variables}\label{section_inf_products}

 
 In this section we will study infinite products of independent beta random variables. These results 
will be used in section \ref{section_exp_func_meromorphic}, in order to describe the distribution of the exponential functional $I_q$.
 The results could also be of independent interest, as similar finite products of beta random variables appear in other areas of probability, for example, in the definition of the Poisson-Dirichlet distribution, 
see \cite{Pitman_Yor}. 
 
For $a,b>0$, let $B_{(a,b)}$ denote the beta random variable, having distribution
\beqq
\p ( B_{(a,b)} \in \d x)=\frac{\Gamma(a+b)}{\Gamma(a) \Gamma(b)} x^{a-1} (1-x)^{b-1} \d x, \;\;\; 0<x<1.
\eeqq 
With any two unbounded sequences ${\alpha}=\{\alpha_n\}_{n\ge 1}$ and 
${\beta}=\{\beta_n\}_{n\ge 1}$  which satisfy the interlacing property
  \beq\label{interlacing_property}
 0<\alpha_1 < \beta_1 < \alpha_2 < \beta_2 < \alpha_3 <\beta_3  \; \dots
 \eeq
  we associate an infinite product of independent beta random variables, defined as 
\beq\label{def_X}
 X(\alpha,\beta):=\prod\limits_{n\ge 1} B_{(\alpha_n,\; \beta_n-\alpha_n)} \frac{\beta_n}{\alpha_n}.
\eeq
The random variable $X(\alpha,\beta)$ is the main object of interest in this section. Our first task is to establish that 
this random variable is well-defined.

\begin{proposition}\label{prop_infinite_product}
 The infinite product in \eqref{def_X} converges a.s.
\end{proposition}

Before proving Proposition \ref{prop_infinite_product}, let us establish the following simple (but useful) result. 
\begin{lemma}\label{lemmaMain}
Assume that $\alpha$ and $\beta$ are two unbounded sequences satisfying \eqref{interlacing_property}, and $f: \r^+ \mapsto \r$ is a monotone function such that $\lim_{x\to +\infty} f(x)=0$. Then 
\beq\label{Lemma_main_eq1}
\left | \sum\limits_{n\ge 1} (f(\beta_n)-f(\alpha_n)) \right | < |f(\alpha_1)|. 
\eeq
\end{lemma} 
\begin{proof}
Assume that $f$ is increasing. Then the condition $f(+\infty)=0$ implies that $f(x)\le 0$ for all $x$, thus for any $m \in {\mathbb N}$
we have
\beqq
0\le S_m := \sum\limits_{n=1}^m (f(\beta_n)-f(\alpha_n)) 
\le
 \sum\limits_{n=1}^m (f(\alpha_{n+1})-f(\alpha_n))=f(\alpha_{m+1})-f(\alpha_1)<
 -f(\alpha_1)=|f(\alpha_1)|. 
\eeqq
The above inequality and the fact that the sequence $S_m$ is increasing show that the series  
in \eqref{Lemma_main_eq1} converges and its sum is bounded by $|f(\alpha_1)|$. The case when $f$ is decreasing can be considered in exactly the same way. 
\end{proof}

\noindent
{\it Proof of Proposition \ref{prop_infinite_product}:}
Considering the logarithm of both sides of \eqref{def_X}, we see that we need to establish the a.s. convergence of the infinite series 
\beq\label{Xc_series}
\log( X(\alpha,\beta))=\sum\limits_{n\ge 1} \ln\left(B_{(\alpha_n,\beta_n-\alpha_n)} \frac{\beta_n}{\alpha_n} \right).
\eeq 
The Mellin transform of a beta random variable is given by
\beq\label{Mellin_beta}
\e\left[ \left(B_{(a,b)}\right)^{s-1} \right]=\frac{\Gamma(a+b)\Gamma(a+s-1)}{\Gamma(a) \Gamma(a+b+s-1)},
\;\;\; \re(s)>1-a. 
\eeq
By differentiating the above identity twice and setting $s=1$, we find
\beqq
\e\left[\ln(B_{(a,b)} )\right]=\psi(a)-\psi(a+b), \;\;\;
{\textnormal{Var}}\left[\ln(B_{(a,b)} )\right]=\psi'(a)-\psi'(a+b),
\eeqq 
where $\psi(z):=\Gamma'(z)/\Gamma(z)$ is the digamma function. It is known that 
$f(z):=\ln(z)-\psi(z)$ is a completely monotone function which decreases to zero 
(see  \cite[Theorem 1]{Alzer} or formula 8.361.8 in \cite{Jeffrey2007}). This implies 
that the function $-f'(z)+1/z\equiv \psi'(z)$, has the same property. Applying Lemma \ref{lemmaMain} we conclude that both series  
\beqq
\sum\limits_{n\ge 1} \e\left[\ln\left(B_{(\alpha_n,\beta_n-\alpha_n)} \frac{\beta_n}{\alpha_n} \right)\right]
&=&\sum\limits_{n\ge 1} (f(\alpha_n)-f(\beta_n)), \\
\sum\limits_{n\ge 1} {\textnormal{Var}}\left[\ln\left(B_{(\alpha_n,\beta_n-\alpha_n)} \frac{\beta_n}{\alpha_n} \right)\right]
&=&\sum\limits_{n\ge 1} ( \psi'(\alpha_n) - \psi'(\beta_n) ) 
\eeqq
converge, therefore Khintchine-Kolmogorov Convergence Theorem implies a.s. convergence of the infinite series \eqref{Xc_series}. 
\qed

\begin{theorem}\label{theorem_main}
The Mellin transform $\mm(s):=\e[(X(\alpha,\beta))^{s-1}]$ exists for all $s>1-\alpha_1$ and it can be analytically continued to a meromorphic function 
\beq\label{inf_product_MXs}
\mm(s)=\prod\limits_{n\ge 1} \frac{\Gamma(\beta_n)\Gamma(\alpha_n+s-1)}{\Gamma(\alpha_n) \Gamma(\beta_n+s-1)}
\left(\frac{\beta_n}{\alpha_n}\right)^{s-1}.
\eeq
 The above infinite product converges uniformly on compact subsets of the complex plane, not containing the poles of 
 $\mm(s)$. 
\end{theorem}
\begin{proof}
Proposition \ref{prop_infinite_product} combined with \eqref{Mellin_beta} and L\'evy's Continuity Theorem imply that 
\eqref{inf_product_MXs} is true for all $s\in \c$ on the vertical line $\re(s)=1$, and that the infinite product in the right-hand side of 
 \eqref{inf_product_MXs} converges uniformly on compact subsets of this line. Our first goal is to prove uniform convergence of this product on any compact subset of $\c$, excluding the poles of $\mm(s)$.

For $t>0$ and $a,z\in \c$ satysfying $\re(z)>\max(0,-\re(a))$, let us define
\beq\label{def_faz}  
  f_a(z):=\ln \left[\frac{\Gamma(a+z)}{\Gamma(z)z^a}\right], \;\;\; 
  g_a(t):=\left(a-\frac{1-e^{-at}}{1-e^{-t}}\right) \frac{1}{t}. 
\eeq 
It is known that for $a\in \r$, the function  $z \in (\max(0,-a), \infty) \mapsto |f_a(z)|$ is completely monotone. This follows 
from \cite[Theorem 4]{QiNiuChen} if $a>0$ and from  \cite[Theorem 1]{Qi2008281} if $a<0$. We also have the following integral representation,
\beq\label{faz_Laplace_transform}
f_a(z)=\int\limits_0^{\infty} g_a(t) e^{-z t} \d t, \;\;\; a,z\in \c, \; \re(z)>\max(0,-\re(a)). 
\eeq
which can be established using formula 8.361.5 in \cite{Jeffrey2007}.

Let $A$ be a compact subset of $\c$. Define 
\beqq
v^-=\min\{\re(s) : \; s\in A\}, \;\;\;
v^+=\max\{ \re(s) : \; s \in A\}.
\eeqq
The convergence of the infinite product \eqref{inf_product_MXs} is not affected by any finite number of terms.
Therefore, without the loss of generality we can assume that $1-\alpha_1<v^-$. If this was not true, 
we would take $N$ large enough, so that $1-\alpha_{N}<v^-$, and 
investigate the convergence of the tail $\prod_{n\ge N}(\dots)$ of the infinite product in \eqref{inf_product_MXs} .

For $k\ge 1$ we define 
\beqq
\mm_{k}(s):=\prod\limits_{n=1}^k \frac{\Gamma(\beta_n)\Gamma(\alpha_n+s-1)}{\Gamma(\alpha_n) \Gamma(\beta_n+s-1)}
\left(\frac{\beta_n}{\alpha_n}\right)^{s-1}.
\eeqq
The function $\mm_{k}(s)$ is analytic and zero-free in the half-plane $\re(s)>1-\alpha_{1}$,
and due to our assumption $1-\alpha_1<v^-$, this half-plane includes the set $A$. Using \eqref{faz_Laplace_transform} and Lemma \ref{lemmaMain} we find that for all 
$l>k\ge 1$ and $s\in A$, 
\beqq
\left| \ln(\mm_{l}(s))-\ln(\mm_k(s)) \right|&=& \left | \sum\limits_{n=k+1}^{l} (f_{s-1}(\alpha_n)-f_{s-1}(\beta_n)) \right |=
\left | \int\limits_0^{\infty} g_{s-1}(t) \sum\limits_{n=k+1}^{l} \left( e^{-\alpha_n t} -e^{-\beta_n t} \right) \d t \right | \\
&\le& \int\limits_0^{\infty}  |g_{s-1}(t)| \sum\limits_{n=k+1}^{\infty} \left( e^{-\alpha_n t} -e^{-\beta_n t} \right) \d t
< \int\limits_0^{\infty}  |g_{s-1}(t)| e^{-\alpha_{k+1} t} \d t. 
\eeqq 
For any fixed $s\in A$, the right-hand side in the above inequality can be made arbitrarily small if $k$ is sufficiently large, since $\alpha_{k}  \to +\infty$ as $k\to +\infty$. This shows that 
$\mm_k(s)$ forms a Cauchy sequence and, therefore, the infinite product in 
\eqref{inf_product_MXs} converges pointwise. 

Now we need to establish the uniform convergence on the compact set $A \subset \c$. Using \eqref{Mellin_beta},
we check that for all $s$ in the half-plane $\re(s)>1-\alpha_1$ we have $\mm_k(s)\equiv\e[X_k^{s-1}]$, where $X_k$ is defined by
\beqq
X_k:=\prod\limits_{n=1}^k B_{(\alpha_n,\beta_n-\alpha_n)} \frac{\beta_n}{\alpha_n}. 
\eeqq
Denoting $v=\re(s)$ and using Lemma \eqref{lemmaMain} and monotonicity of 
the function $z \in (\max(0,-a),\infty)  \mapsto f_a(z)$ for $a\in \r$ we obtain
\beq\label{uniform_bound}
|\mm_k(s)|=|\e[X_k^{s-1}]| &\le& \e[ |X_k^{s-1}|]=\e[ X_k^{v-1}]=\mm_k(v)\\ \nonumber
&=&\exp\left( \sum\limits_{n=1}^k ( f_{v-1}(\alpha_n)-f_{v-1}(\beta_n)) \right)
< \exp(|f_{v-1}(\alpha_{1})|). 
\eeq
The function $v\mapsto f_{v-1}(\alpha_{1})$ is continuous on the interval $v\in [v^-, v^+]$, therefore it is bounded on this interval. 
This fact combined with the inequality \eqref{uniform_bound} implies uniform boundedness of all functions $\mm_k(s)$, $k\ge 1$ 
in the vertical strip $v^- \le \re(s) \le v^+$, therefore also on the set $A$. 
By the Vitaly-Porter Theorem, uniform boundedness and pointwise convergence of $\mm_k(s)$ imply uniform convergence of these functions. This proves that the product \eqref{inf_product_MXs} converges uniformly to a function which is analytic in $\re(s)>1-\alpha_1$. By analytic continuation we conclude that the Mellin transform of $X(\alpha,\beta)$ exists everywhere in this half-plane.      
\end{proof}

In the next proposition we show that the Mellin transform of $X(\alpha,\beta)$ satisfies two important identities, which will be crucial for our results on exponential functionals of meromorphic processes in section \ref{section_exp_func_meromorphic}. 

\begin{proposition}\label{propeties_of_Ms}
 Assume that $\alpha_n$ and $\beta_n$ satisfy the interlacing property \eqref{interlacing_property}. 
\mbox{}
\begin{itemize}
\item[(i)] For $n\ge 1$ define $\tilde \alpha_n=\beta_n$ and $\tilde \beta_n=\alpha_{n+1}$. The sequences $\tilde \alpha_n$ and $\tilde \beta_n$ also satisfy the interlacing property  \eqref{interlacing_property}. Define $\tilde \mm(s)=\e\left[X(\tilde \alpha,\tilde\beta)^{s-1}\right]$. We have the following identity  
\beq\label{idenity_one_over_M}
\mm(s) \times \tilde \mm(s) =\frac{\Gamma(\alpha_1+s-1)}{\Gamma(\alpha_1)\alpha_1^{s-1}}, \;\;\;
\re(s)>1-\alpha_1. 
\eeq
\item[(ii)] The function $\mm(s)$ satisfies $\mm(s+1)=\phi(s)\mm(s)$ for all $s\in \c$, where $\phi(s)$ is a meromorphic function defined as 
\beq\label{def_phi}
\phi(s):=\prod\limits_{n\ge 1}\frac{1+\frac{s-1}{\alpha_n}}{1+\frac{s-1}{\beta_n}}, \;\;\; s\in \c. 
\eeq
\end{itemize}
\end{proposition}
\begin{proof}
The proof of (i) follows at once from \eqref{inf_product_MXs}. 
Lemma \ref{lemmaMain} implies convergence of the infinite product in \eqref{def_phi}, and  the proof of (ii) follows easily from \eqref{inf_product_MXs} and identity $\Gamma(s+1)=s\Gamma(s)$.   
\end{proof}


\section{Exponential functional of meromorphic processes}\label{section_exp_func_meromorphic}


Let us review some important properties of meromorphic processes, which we will need for our further results (see \cite{KuzKyPa2011} for more details). The density of the L\'evy measure of a meromorphic process is defined  by \eqref{def_pi}. It is clear that the L\'evy measure has exponentially decreasing tails. The only restriction on the coefficients determining the L\'evy measure is the convergence of the 
series 
\beqq
\sum\limits_{n\ge 1} \left(\frac{a_n}{\rho_n^2}+\frac{\hat a_n}{\hat \rho_n^2} \right)<+\infty,
\eeqq
which is easily seen to be equivalent to the convergence of the integral $\int_{\r} x^2 \pi(x) \d x$. 

From the Levy-Khintchine formula we find that 
the Laplace exponent $\psi(z):=\ln \e [\exp(z X_1)]$ of the meromorphic process $X$ is given by 
 \beq\label{eq_psi}
 \psi(z)=\frac 12 \sigma^2 z^2 + \mu z
 +z^2 \sum\limits_{n\ge 1}\frac{a_n }{\rho_n(\rho_n- z)}+z^2 \sum\limits_{n\ge 1} \frac{\hat a_n}{\hat \rho_n (\hat \rho_n+z)}
, \;\;\; -\hat \rho_1< \re(z) <\rho_1,
 \eeq
 where $\sigma\ge 0$ and $\mu \in \r$. 
The function $\psi(z)$ can be analytically continued to a real meromorphic function, having only simple poles at points $\rho_n$, $-\hat \rho_n$. The most important analytical property of meromorphic processes 
(also satisfied by hyper-exponential processes) is that for any $q>0$ the equation $\psi(z)=q$ has only simple real solutions
$\zeta_n$, $-\hat \zeta_n$, which satisfy the interlacing property  
 \beq\label{interlacing_prop}
 ... -\hat\rho_2 <-\hat\zeta_2<-\hat\rho_1<- \hat\zeta_1 <0 < \zeta_1 <  \rho_1 < \zeta_2 < \rho_2 < ...
 \eeq
 When we need to emphasize the dependence of these solutions on the parameter $q$, we will denote them
 $\zeta_n(q)$ and $\hat \zeta_n(q)$. 
 Moreover, we have the following infinite product representation
  \beq\label{WH_fact_1}
 q-\psi(z)=q\prod\limits_{n\ge 1} \frac{1-\frac{z}{\zeta_n}}{1-\frac{z}{\rho_n}} \times
\prod\limits_{n\ge 1} \frac{1+\frac{z}{\hat \zeta_n}}{1+\frac{z}{\hat \rho_n}}, \;\;\; z\in \c,
 \eeq
 where the two infinite products converge due to Lemma \ref{lemmaMain}. Formula \eqref{WH_fact_1} is the key 
 result in the theory of meromorphic processes, since it allows to identify the Wiener-Hopf factors for the process $X$ (see \cite[Theorem 2]{KuzKyPa2011}). 
 
 The next theorem is the main theoretical result in this paper. We identify the distribution of the exponential functional $I_q$ for a general meromorphic process. This result should be compared with \cite[Theorem 4.3]{CaiKou2010}, where the distibution of the exponential functional is identified for a hyper-exponential process having a non-zero Gaussian component.

\begin{theorem}\label{thm_exp_functional} Assume that $q>0$.  Define $\hat\rho_0=0$ and the four sequences 
\beqq
\zeta=\{\zeta_n\}_{n\ge 1}, \; \rho=\{\rho_n\}_{n\ge 1}, \; 
\tilde  \zeta=\{1+\hat\zeta_n\}_{n\ge 1}, \; \tilde \rho=\{1+\hat \rho_{n-1} \}_{n\ge 1}.
\eeqq
Then we have the following identity in distribution
 \beq\label{Iq_distribution_identity}
 I_q \stackrel{d}{=} C  \times  \frac{X(\tilde \rho, \tilde \zeta)}{X(\zeta,\rho)},
 \eeq
 where 
 \beq\label{def_C}
C=C(q):=q^{-1} \prod\limits_{n\ge 1} \frac{1+\frac{1}{\hat \rho_n}}{1+\frac{1}{\hat \zeta_n}},
\eeq
and the random variables  $X(\tilde \rho, \tilde \zeta)$ and $X(\zeta,\rho)$ are independent
and are defined by \eqref{def_X}. 
The Mellin transform $\mm(s)=\mm(s,q):=\e[I_q^{s-1}]$ is finite for $0<\re(s)<1+\zeta_1$ and is given by
 \beq\label{def_f}
 \mm(s)=
C^{s-1} 
 \prod\limits_{n\ge 1} 
 \frac{ \Gamma(\hat \zeta_{n}+1) \Gamma(\hat \rho_{n-1}+s)}
{ \Gamma(\hat \rho_{n-1}+1)\Gamma(\hat \zeta_{n}+s)}
\left(\frac{\hat \zeta_{n}+1}{ \hat \rho_{n-1}+1} \right)^{s-1}
\frac{\Gamma(\rho_n)\Gamma(\zeta_n+1-s)  }
{\Gamma(\zeta_n) \Gamma(\rho_n+1-s) }
\left(\frac{\zeta_n }{\rho_n } \right)^{s-1}.
\eeq
\end{theorem}
 \begin{proof}
First, we note that the two pairs of sequences ($\zeta, \rho$) and ($\tilde \rho,\tilde \zeta$) satisfy the interlacing property
\eqref{interlacing_property}, therefore the random variables $X(\zeta,\rho)$ and $X(\tilde \rho, \tilde \zeta)$ are well-defined. The 
infinite product defining constanct $C$ converges due to Lemma \ref{lemmaMain}. 

Let $f(s)$ denote the function in the right-hand side of \eqref{def_f}.
Define $M_1(s):=\e[(X(\tilde \rho,\tilde \zeta))^{s-1}]$ and $M_2(s):=\e[(X(\zeta,\rho))^{s-1}]$. Using Theorem \ref{theorem_main} we find that $f(s)=C^{s-1} M_1(s) M_2(2-s)$ is the Mellin transform of the random variable
in the right-hand side of \eqref{Iq_distribution_identity}, 
 and that $f(s)$ is analytic in the strip $0<\re(s)<1+\zeta_1$. Formula \eqref{def_f} 
 implies that $f(s)$ is zero-free in the wider strip $-\hat \zeta_1<\re(s)<1+\rho_1$.

 We will prove that $\mm(s)\equiv f(s)$, which would imply the identity in distribution  
\eqref{Iq_distribution_identity}. We will use the verification result, see \cite[Proposition 2]{KuPa2010_prep}, which states that 
$\mm(s)\equiv f(s)$, provided that the following three conditions are satisfied: 
\begin{itemize} 
 \item[(i)] for some $\theta>0$, the function $f(s)$ is analytic and zero free in the vertical strip $0<\re(s)<1+\theta$;
 \item[(ii)] the function $f(s)$ satisfies
 \beqq
 f(s+1)=\frac{s}{q-\psi(s)} f(s), \;\;\; 0<s<\theta,
 \eeqq
 where $\psi(s)$ is the Laplace exponent of the process $X$;
 \item[(iii)] $|f(s)|^{-1}=o(\exp(2\pi |\im(s)|))$ as $\im(s)\to \infty$, uniformly in the strip $0<\re(s)<1+\theta$. 
 \end{itemize}

According to the above discussion of the analytic properties of $f(s)$, condition (i) is satisfied with $\theta=\zeta_1$. Let us verify condition (ii). Proposition \ref{propeties_of_Ms}(ii) gives us two identities
\beqq
M_1(s+1)=M_1(s) \times \prod\limits_{n\ge 1} \frac{1+\frac{s-1}{1+\hat \rho_{n-1}}}{1+\frac{s-1}{1+\hat \zeta_n}}, \;\;\;
M_2(1-s)=M_2(2-s) \times \prod\limits_{n\ge 1} \frac{1-\frac{s}{\rho_{n}}}{1-\frac{s}{\zeta_n}}, \;\;\;
s\in \c.  
\eeqq
Combining \eqref{def_f} with the above identities we obtain
\beqq
f(s+1)=C^s M_1(s+1) M_2(1-s)=f(s) \times q^{-1} \prod\limits_{n\ge 1} \frac{1+\frac{1}{\hat \rho_n}}{1+\frac{1}{\hat \zeta_n}}
\times \frac{1+\frac{s-1}{1+\hat \rho_{n-1}}}{1+\frac{s-1}{1+\hat \zeta_n}} 
\times
\frac{1-\frac{s}{\rho_{n}}}{1-\frac{s}{\zeta_n}}. 
\eeqq
where we have also used \eqref{def_C}. Rearranging the terms in the above infinite product and using 
\eqref{WH_fact_1} we conclude that 
\beqq
f(s+1)=f(s) \times s q^{-1} \prod\limits_{n\ge 1} 
\frac{1+\frac{s}{\hat \rho_{n}}}{1+\frac{s}{\hat \zeta_n}} 
\times
\frac{1-\frac{s}{\rho_{n}}}{1-\frac{s}{\zeta_n}}=\frac{s}{q-\psi(s)} f(s),
\eeqq
thus condition (ii) is also satisfied.

Finally, let us show that condition (iii) holds true. We use Proposition \ref{propeties_of_Ms}(i) and find that
\beq\label{eq_f_M1_M2}
f(s)^{-1}=C^{1-s} \frac{\Gamma(\zeta_1)\zeta_1^{1-s}}{\Gamma(s)\Gamma(\zeta_1+1-s)}
\times \tilde M_1(s) \times \tilde M_2(2-s), \;\;\; 0< \re(s)< 1+\zeta_1,
\eeq
where $\tilde M_1(s):=\e[(X(\alpha,\beta))^{s-1}]$ and $\tilde M_2(s):=\e[(X(\tilde \alpha, \tilde \beta))^{s-1}]$ and the sequences $\alpha$, $\beta$, $\tilde \alpha$, $\tilde \beta$ are defined as
follows
\beqq
\alpha_n:=1+\hat\zeta_n, \; \beta_n:=1+\hat \rho_n, \; \tilde \alpha_n:=\rho_n, \; \tilde \beta_n:=\zeta_{n+1}, \;\;\; n\ge 1. 
\eeqq 
According to Theorem \ref{theorem_main}, the Mellin transform $\tilde M_1(s)$ (resp. $\tilde M_2(s)$) is finite 
in the half-plane $\re(s)>-\hat \zeta_1$ (resp. $\re(s)>1-\rho_1$). 
Since $|\tilde M_i(s)|<\tilde M_i(\re(s))$, 
we see that for any $\epsilon>0$, the function 
$\tilde M_1(s)$ (resp. $\tilde M_2(2-s)$) is uniformly bounded in the half-plane $\re(s)\ge \epsilon-\hat \zeta_1$ (resp. $\re(s)\le 1+\rho_1-\epsilon$). Taking 
$\epsilon=\tfrac{1}{2} \min(\hat \zeta_1, \rho_1-\zeta_1)$ we conclude that 
the function $\tilde M_1(s) \times \tilde M_2(2-s)$ is uniformly bounded in the vertical strip $0\le \re(s) \le 1+\zeta_1$. 

To estimate the gamma functions in \eqref{eq_f_M1_M2}, we use formula 8.328.1 in \cite{Jeffrey2007} 
\beqq
\lim\limits_{y\to \infty} |\Gamma(x+\i y)| e^{\frac{\pi}{2}|y|} |y|^{\frac{1}{2}-x} =\sqrt{2\pi}, \;\;\; x, y \in \r. 
\eeqq
It is known that the limit exists uniformly in $x$ on compact subsets of $\r$ (as can be easily seen from Stirling's asymptotic formula for the gamma function). 
The above formula shows that for any $\epsilon>0$ 
\beqq
\frac{1}{|\Gamma(s)\Gamma(\zeta_1+1-s)|}= o(\exp((\pi+\epsilon)|\im(s)|))
\eeqq
as $\im(s) \to \infty$, uniformly in the strip $0\le \re(s) \le 1+\zeta_1$. This fact combined with 
\eqref{eq_f_M1_M2} and uniform boundedness of $\tilde M_1(s)\times \tilde M_2(2-s)$ shows that condition (iii) is also satisfied. Therefore, according to \cite[Proposition 2]{KuPa2010_prep} we have $\mm(s)\equiv f(s)$, 
and this ends the proof of Theorem \ref{thm_exp_functional}. 
 \end{proof}

The results of Theorem  \ref{thm_exp_functional} may be useful for studying 
 self-similar Markov processes. It is well-known that 
 Lamperti transformation \cite{LA} maps a positive self-similar Markov
process $Y$ to a L\'evy process $X$, and the exponential functional of $X$ plays a very important role in describing various properties of the original process $Y$. The last several years have witnessed a large volume of research 
on self-similar Markov process, which are constructed from stable L\'evy processes by conditioning 
or various path transformations (see  \cite{CC,CPP,KuPa2010_prep,KyPaWa}).
We note that in all of these examples (at least in dimension one), the Lamperti transformed process belongs to the family of meromorphic processes. 
Therefore, we hope that our results on exponential functionals of a general meromorphic process may be useful in the future for studying other interesting self-similar Markov processes.

 \begin{remark} The results of Theorem  \ref{thm_exp_functional} can be easily extended to the boundary case $q=0$. The exponential
 functional $I_0=\int_0^{\infty} \exp(X_t) \d t$ is well-defined if $\e[X_1]<0$ (see \cite[Theorem 1]{BYS}). One can check that condition $\psi'(0)=\e[X_1]<0$ implies 
 $\hat \zeta_1(q) \to \zeta_1(0)=0$ as $q\to 0^+$, moreover  $q/\hat \zeta_1(q)\to |\e[X_1]|$. 
 Thus the constant $C=C(q)$ defined by \eqref{def_C} converges as $q\to 0^+$ to 
 \beqq
 C(0)=\frac{1}{|\e[X_1]|} \prod\limits_{n\ge 1} \frac{1+\frac{1}{\hat \rho_n}}{1+\frac{1}{\hat \zeta_{n+1}(0)}}. 
 \eeqq
 Note that $\hat \zeta_{n+1}(0)>\hat \rho_n>0$ for $n\ge 1$, thus the above product is well-defined, it converges
 due to Lemma \ref{lemmaMain}.
  The random variables $X(\tilde \rho, \tilde \zeta)$ and $X(\zeta,\rho)$ are also well-defined in the limit $q\to 0^+$, provided that we
  identify  $B_{(1,0)}\stackrel{d}{=} 1$.
 \end{remark}
 
  \begin{remark} 
The factorization of the exponential functional as a product of two independent random variables  in \eqref{Iq_distribution_identity} is in fact a very general phenomenon. It turns out that in many cases the exponential functional $I_q$ has the same distribution as a product of an exponential functional of a negative of a subordinator and an independent exponential functional of a spectrally positive process, both of these processes being related to the Wiener-Hopf factors of the original L\'evy process $X$, see 
\cite{Patie_Pardo_Savov,Patie_Savov}. Also, we would like to point out that recently Patie and Savov 
\cite{Patie_Savov2} have obtained some very strong and general results on the Mellin transform of the exponential functional. In particular, they show that the Mellin transform can be obtained as a generalized Weierstrass product in terms of the Wiener-Hopf factors of the process $X$. Since the Wiener-Hopf factors of the meromorphic process are known to be infinite products of linear factors (see formula \eqref{WH_fact_1} and
\cite[Theorem 2]{KuzKyPa2011}), these results could lead to an alternative proof of Theorem \ref{thm_exp_functional}. The Mellin transform of $I_q$ can be expressed 
as a double infinite product, interchanging the order in this product and using Weierstrass infinite product representation
for the gamma function one could obtain formula \eqref{def_f}, probably at the cost of considerable technical details.  
 \end{remark}


\section{Numerical examples}\label{section_numerics}


For our numerical examples we will consider a theta process, defined by the Laplace exponent 
\beq\label{def_theta}
\psi(z)=\frac{\sigma^2}{2} z^2 + \mu z +\gamma +(-1)^j &\bigg[ &c_1 \pi  \left(\sqrt{(\alpha_1-z)/\beta_1}\right)^{2j-1} \coth \left( \pi  \sqrt{(\alpha_1-z)/\beta_1} \right) \\
&+&c_2 \pi  \left(\sqrt{(\alpha_2+z)/\beta_2}\right)^{2j-1} \coth \left( \pi  \sqrt{(\alpha_2+z)/\beta_2} \right)\bigg].
\eeq
These processes were introduced in \cite{Kuz2010b} as examples of meromorphic processes for which the Laplace exponent can be computed in a particularly simple form. The parameter $j$ can take two values $j \in \{1,2\}$, the parameter $\gamma$ is always chosen so that 
$\psi(0)=0$. Other restrictions on the parameters are 
\beqq
\sigma\ge 0, \; \mu \in \r, \; c_i \ge 0, \; \alpha_i\ge 0, \; \beta_i>0. 
\eeqq
As was shown in \cite{Kuz2010b}, the density of the L\'evy measure of a theta process is given by \eqref{def_pi}, with parameters 
$a_n$, $\hat a_n$, $\rho_n$, $\hat \rho_n$ defined as follows
\beq\label{def_rho_a}
\rho_n=\alpha_1+\beta_1 n^2, \; \hat \rho_n=\alpha_2+\beta_2 n^2, \; a_n \rho_n=2 c_1 \beta_1 n^{2j}, \;
\hat a_n \hat \rho_n=2 c_2 \beta_2 n^{2j}\hat \rho_n, \;\;\; n\ge 1.   
\eeq

For our numerical experiments, we will consider the following two parameter sets  
\beq\label{def_parameter_sets}
&&\textnormal{Parameter set 1:}  \; \; j=1, \; \sigma=0.1, \\ \nonumber
&&\textnormal{Parameter set 2:}  \; \; j=2, \; \sigma=0.0,  
\eeq
with the remaining parameters being fixed at
\beqq
\mu=0.1, \; c_1=0.15, \; c_2=0.3, \; \alpha_1=\alpha_2=1.5, \; \beta_1=\beta_2=2. 
\eeqq
The parameter set 1 defines a process with a non-zero Gaussian component and jumps of infinite activity but finite variation, while the
parameter set 2 defines a process with zero Gaussian component and jumps of infinite variation (see \cite[Proposition 4]{Kuz2010b}).

 
\subsection{Approximating the Mellin transform of the exponential functional}\label{subsection_distribution_Iq}

 
First we will discuss the problem of computing the Mellin transform and the density of the exponential functional. This algorithm will be useful later for our discussion on pricing of Asian options. 
We assume that $q>0$ and denote the density of $I_q$ by $p(x)$. 
 We start by expressing $p(x)$ as an inverse Mellin transform of
 $\mm(s)=\e[I_q^{s-1}]$
\beq\label{p_Mellin_transform}
p(x)=\frac{x^{-c}}{2\pi} \int\limits_{\r} \mm(c+\i u) e^{-\i u \ln(x)} \d u,
\eeq 
where $c$ can be any number in the interval $(0,1+\zeta_1)$. 
We see that there are two main issues in computing $p(x)$.  In order to compute the Fourier integral in the right-hand side of \eqref{p_Mellin_transform}, we will use Filon's method.   
This  algorithm is well-known and we will refer the reader to \cite{Filon, Fosdick, Iserles}. 
However, before we can apply Filon's method, we need to be able to compute $\mm(s)$ (given by  an 
infinite product 
\eqref{def_f}) to a reasonably high precision. 

A naive way of computing the Mellin transform $\mm(s)$  would be to simply 
truncate the infinite products in 
\eqref{def_C} and \eqref{def_f}  at $n=N$. After rearranging the terms, this would give us an approximation 
 \beq\label{def_Mn}
\mm_N(s):=
a_N \times b_N^{s-1} \times
 \prod\limits_{n=1}^N
 \frac{\Gamma(\hat \rho_{n-1}+s)}
{\Gamma(\hat \zeta_{n}+s)}
\frac{\Gamma(\zeta_n+1-s)}
{\Gamma(\rho_n+1-s)}
 \eeq
 where
 \beqq
 b_N:=\frac{1+\hat \rho_N}{q} \prod\limits_{n=1}^{N} \frac{\zeta_n\hat \zeta_n}{\rho_n\hat \rho_n}
 \eeqq
 and $a_N$ is a normalizing constant, uniquely determined by the condition $\mm_N(1)=1$. 
 According to Theorem \ref{thm_exp_functional}, $\mm_N(s) \to \mm(s)$ as $N\to +\infty$. The problem is that the convergence could be quite slow, and we need to find a way to accelerate it. 
 
 Let us define the remainder term $R_N(s)=\mm(s)/\mm_N(s)$. It is clear hat $R_N(s)$ 
 is the tail of the infinite product in  \eqref{def_f}. Instead of replacing $R_N(s)$ by 
 1, let us try to do a better job of approximating this function. 
 
  Using Theorems  \ref{theorem_main} and \ref{thm_exp_functional} one can see that 
  $R_N(s)$ is the Mellin transform of some random variable, which we will denote by $\epsilon^{(N)}$. 
  This random variable can be obtained by taking the tail in the infinite products \eqref{def_X} and \eqref{def_C}, however the exact form of $\epsilon^{(N)}$ is not important. What is important is that we can compute exactly the moments  $m_k:=\e[(\epsilon^{(N)})^k]$. 
  Note that the function $R_N(s)=\e[(\epsilon^{(N)})^{s-1}]$ is analytic in the strip $-\hat \rho_N < \re(s)< 1+\zeta_{N+1}$, therefore the moments $m_k$ are finite for all $k<1+\zeta_{N+1}$. 
  Using the functional equation $\mm(s+1)=s\mm(s)/(q-\psi(s))$, 
  we find
 \beq\label{moments_epsilon_N} 
 m_k=\e[(\epsilon^{(N)})^k]=R_N(k+1)=\frac{\mm(k+1)}{\mm_N(k+1)}=
 \frac{k!}{\mm_N(k+1)} \prod\limits_{j=1}^{k} \frac{1}{q-\psi(j)},
 \eeq 
 and these quantities can be readily computed (in case if one of $\psi(j)$ is equal to zero or infinity, one can still compute $m_k$ via L'H\^ospital's rule).
 
 Our plan is to replace $\epsilon^{(N)}$ (whose distribution we can not compute explicitly) by a simple random variable $\xi$, with known distribution and Mellin transform, so that the first two moments of $\xi$ match the first two moments of $\epsilon^{(N)}$. We will take $\xi$ to be a beta random variable of the second kind, which is defined by its density
 \beqq
 \p(\xi \in \d x)=\frac{\Gamma(a)\Gamma(b)}{\Gamma(a+b)} y^{a-1}(1+y)^{-a-b} \d y, \;\;\; y>0. 
 \eeqq
 where the parameters $a$ and $b$ must be positive. The Mellin transform of $\xi$ is given by
 \beqq
 \e[\xi^{s-1}]=\frac{\Gamma(a+s-1)\Gamma(b+1-s)}{\Gamma(a)\Gamma(b)}. 
 \eeqq
 One can check that if we define the parameters
 \beq\label{def_a_b}
 a=m_1 \frac{m_1+m_2}{m_2-m_1^2}, \;\;\; b=1+\frac{m_1+m_2}{m_2-m_1^2},
 \eeq
 then we have $\e[\xi]=m_1$ and $\e[\xi^2]=m_2$. 
 
 Let us summarize the algorithm for approximating the Mellin transform $\mm(s)$. We set  $N$ to be a large number (large enough so that the condition $\zeta_{N+1}>1$ is satisfied). Then $m_1$ and $m_2$ are finite, and 
 we compute these numbers using formula \eqref{moments_epsilon_N}. We evaluate the parameters $a$ and $b$ via \eqref{def_a_b} and approximate the Mellin transform $\mm(s)$ by
 \beq\label{Mellin_correction_term}
 \mm(s) \approx \mm_N(s) \frac{\Gamma(a+s-1)\Gamma(b+1-s)}{\Gamma(a)\Gamma(b)}.
 \eeq
 We emphasize, that the right-hand side of \eqref{Mellin_correction_term} is the Mellin transform of a random variable $I_q^{(N)}$, which converges to $I_q$ in distribution as $N\to +\infty$, and which has the same first two moments as $I_q$ if the latter exist, and ``analytically continued" moments if the classical moments do not exist. 

 \begin{figure}
\centering
\subfloat[]
{\label{fig4a}\includegraphics[height =4.5cm]{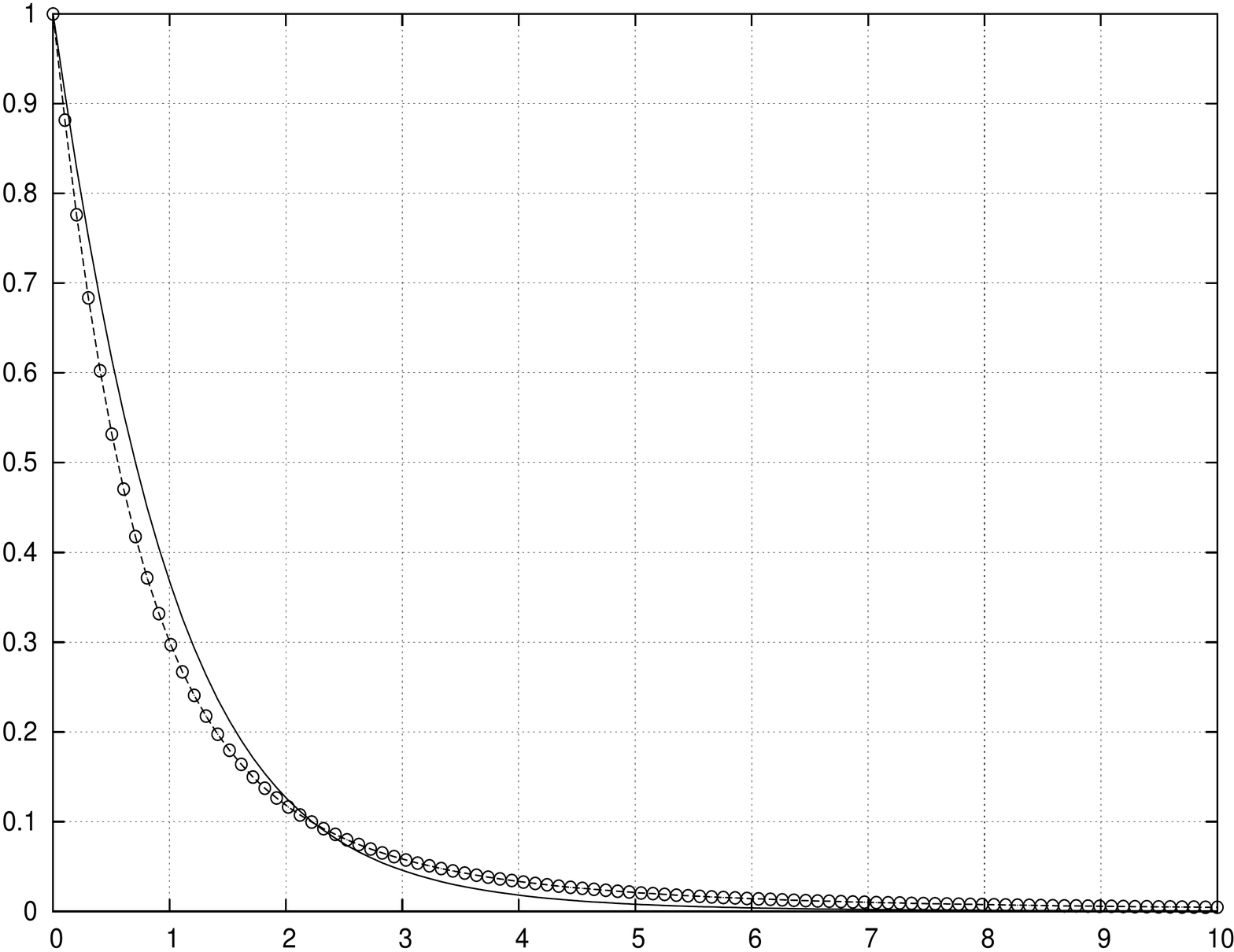}} 
\subfloat[]
{\label{fig4b}\includegraphics[height =4.5cm]{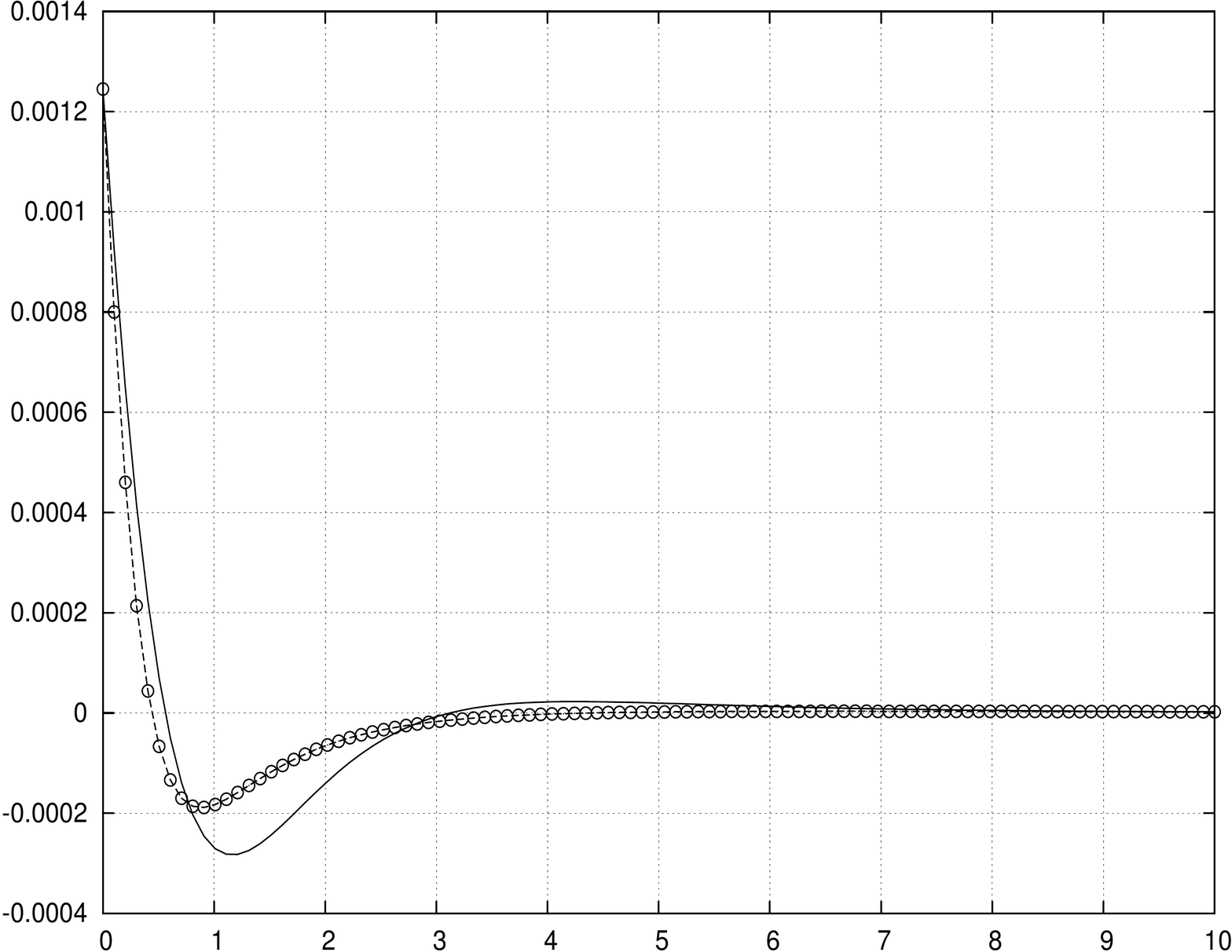}}  
\subfloat[]
{\label{fig4b}\includegraphics[height =4.5cm]{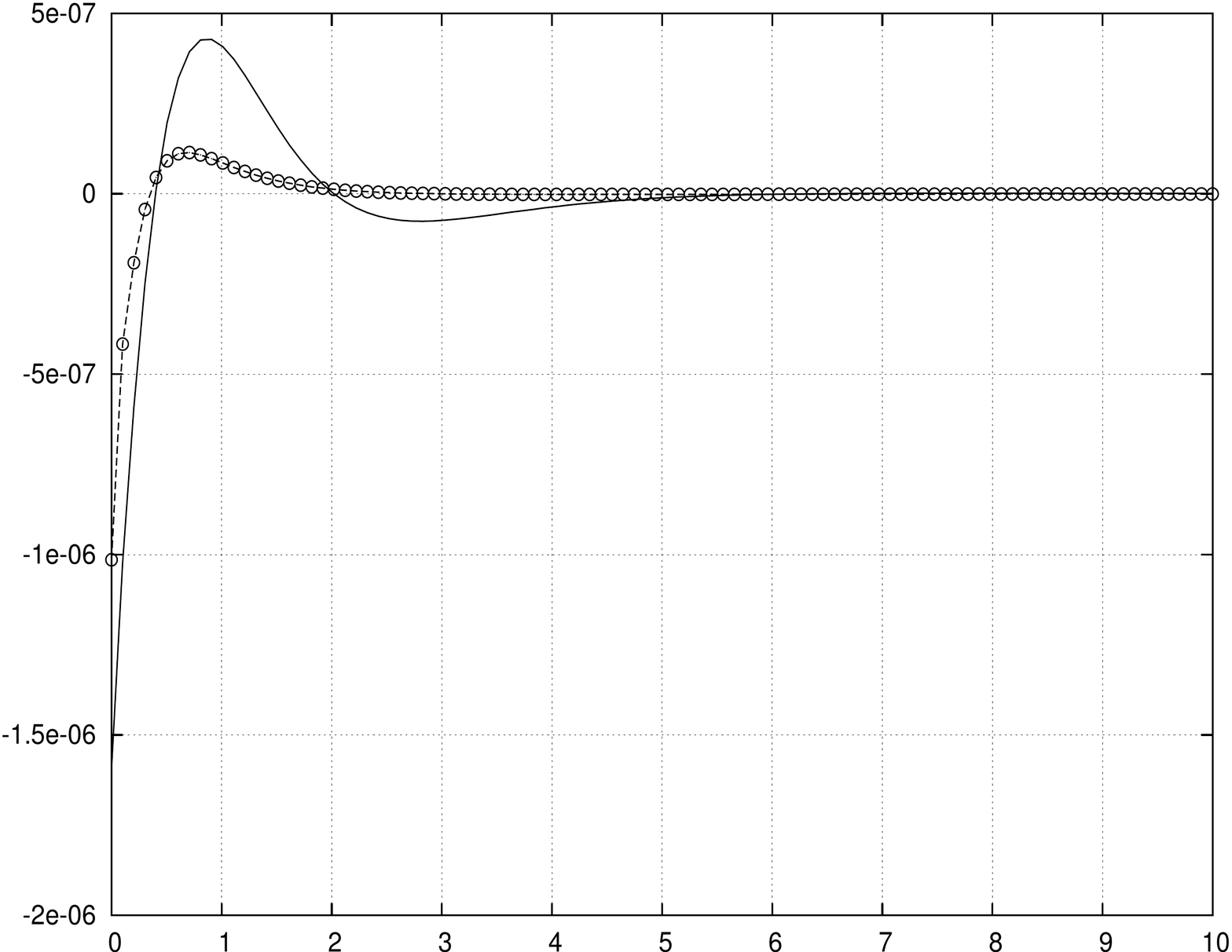}}  
\caption{(a) The density of the exponential functional $I_1$
with $N=400$ (the benchmark).  (b) The error with $N=20$ (no correction). (c) The error with $N=20$ (with correction term). Solid line (resp. circles) represent parameter set I (resp. II).
  } 
\label{fig1}
\end{figure}

We illustrate the efficiency of this approximation by a numerical example. We compute the density of the exponential functional $I_1$ for the two parameter sets \eqref{def_parameter_sets}, 
using approximation \eqref{def_Mn} with  $N=400$. The graphs are shown on figure \ref{fig1}(a), and we will take these results as our benchmark. 
Then we calculate the density (using the same approximation \eqref{def_Mn}) with $N=20$, the error between these results and our benchmark 
is shown on figure \ref{fig1}(b).  We see that the maximum error is of the order $\textnormal{1.0e-3}$. Finally, we perform the same calculation with $N=20$, but now we use the approximation for the Mellin transform
with the correction term \eqref{Mellin_correction_term}. On figure \ref{fig1}(c) we see that
the maximum error is of the order $\textnormal{1.0e-6}$, therefore our approximation  
\eqref{Mellin_correction_term} decreases the error by a factor of 1000, and it seems to be very efficient.


\subsection{Computing the price of  an Asian option}\label{subsection_Asian_options}

We recall that the stock price process is defined as a geometric L\'evy process $S_t=S_0e^{X_t}$ (where $X$ is a L\'evy process started from zero and $S_0$ is the initial stock price). We
are working under risk-neutral measure, so that the process $S_t \exp(-r t)$ is a martingale under measure $\p$). This condition can be expressed in terms of the Laplace exponent of the process $X$ by the equation $\psi(1)=r$ (note that a condition $\rho_1>1$ is necessary to ensure $\psi(1)<+\infty$). In order to satisfy this risk-neutral condition, we will choose accordingly the value of $\mu$ (which is responsible for the linear drift of the process) in formula \eqref{def_theta}. The price of a fixed strike Asian option is given by the expectation 
\eqref{Asian_option}. Let us introduce the function
\beqq
f(k,t):=\e\left[ \left(\int_{0}^t e^{X_u} \d u - k\right)^+\right].
\eeqq
Since $C(S_0,K,T)=\exp(-r T) \times S_0 \times f(K/S_0,T)$, we can study an equivalent problem of computing numerically the values of the function $f(k,t)$.

 Our first goal is to connect the function $f(k,t)$ with the Mellin transform of the exponential functional
 $I_q$. This result was given in 
\cite[Theorem 2.1]{CaiKou2010}, however we reproduce the main steps for the sake of completeness.  We see that 
for any $q>r$ 
\beq\label{hkq_Laplace}
h(k,q):=q\int_0^{\infty} e^{-q t} f(k,t) \d t= 
\e\left[ \left(I_q - k\right)^+\right],
\eeq 
where $\ee(q)$ is an exponential random variable with mean $1/q$ (independent of $X$). Note that $h(k,q)<+\infty$ for $q>r$, since $\e\left[ \left(I_q - k\right)^+\right]<\e[I_q]=(q-r)^{-1}$.
Next, let us check that for $q>r$ we have $\zeta_1>1$. This is clear, since $\zeta_1$ is the smallest positive solution to the equation $\psi(z)=q$, and we have the risk-neutral condition $\psi(1)=r$. Since $\psi(z)$ is convex and $\psi(0)=0$, we find that 
$\psi'(1)>0$, thus the solution to equation $\psi(z)=q$ must be greater than $1$.  Having established that $\zeta_1>1$, we verify that for
$q>r$ and $0<s<\zeta_1-1$ we have
\beq\label{hkq_Mellin}
\int_{0}^{\infty} h(k,q) k^{s-1} \d k&=&\e\left[ \int_0^{\infty} \left(I_q - k\right)^+ k^{s-1} \d k\right]
\\ \nonumber
&=&\e\left[ \int_0^{I_q} \left(I_q - k\right) k^{s-1} \d k\right]
=\frac{\e \left[ I_q^{s+1} \right]}{s(s+1)}=\frac{\mm(s+2,q)}{s(s+1)},
\eeq
where in the first step we have used Fubini's theorem. Note that the right-hand side is finite for $0<s<\zeta_1-1$, since 
$\mm(s)$ is finite in the strip $\re(s) \in (0,1+\zeta_1)$ (see Theorem \ref{thm_exp_functional}). 
The two formulas \eqref{hkq_Laplace} and \eqref{hkq_Mellin} are the foundations for our algorithm to compute 
the price of an Asian option.

\vspace{0.2cm}

\noindent
{\bf Algorithm 1: approximating the Mellin transform of the exponential functional.}

\noindent
This algorithm is based on inverting Laplace and Mellin transform in 
\eqref{hkq_Laplace} and \eqref{hkq_Mellin} and approximating $\mm(s)$ by the algorithm presented in 
section \ref{subsection_distribution_Iq}. First, 
for $d_2>r$ and $q=d_2+\i u$ we compute $h(k,q)$ as the inverse Mellin transfom
\beq\label{hkq_inverse_Mellin}
h(k,q)=\frac{k^{-d_1}}{2\pi} \int\limits_{\r} \frac{\mm(d_1+\i v +2,q)}{(d_1+\i v)(d_1+\i v+1)} e^{-\i v \ln(k)} \d v,
\eeq
where $d_1 \in (0,\zeta_1(d_2)-1)$. Second, we compute $f(k,t)$ as the inverse Laplace transform, which can be rewritten as the cosine transform 
\beq\label{fkt_inverse_Laplace}
f(k,t)=\frac{2 e^{d_2 t}}{\pi}  \int\limits_{0}^{\infty} \re\left[\frac{h(k,d_2+\i u)}{d_2+\i u}\right] \cos(u t) \d u.
\eeq

We set $d_1=d_2=0.25$ and truncate the integral in \eqref{hkq_inverse_Mellin} (resp. \eqref{fkt_inverse_Laplace}) so that the domain of integration is  $-100<v<100$  (resp. $0<u<200$), and use Filon's method
\cite{Filon, Fosdick} with 400 discretization points to evaluate each of these integrals. 
The Mellin transform is computed using the approximation algorithm presented in section
 \ref{subsection_distribution_Iq}, with the Mellin transform truncated at $N$ terms
 (we will set $N\in \{10,20,40,80\}$ in our computations). Computing the Mellin transform requires 
 computing $2N$ solutions to the equation $\psi(z)=q$. An efficient algorithm of how to find these solutions
 for both real and complex values of $q$ was presented in \cite{Kuz2010a}.

We would like to point out that our algorithm requires computing the values of $\mm(s,q)$ for complex values of $q$, while Theorem \ref{thm_exp_functional} was established only for $q>0$. 
To extend the results of Theorem \ref{thm_exp_functional} for complex values of $q$, the main step would be to establish the uniform convergence of the infinite product in the right-hand side of 
 \eqref{def_f} for complex $q$, for which we will need some additional information on the behavior of the solutions 
 $\zeta_n$, $-\hat \zeta_n$ 
 to the equation $\psi(z)=q$ for complex $q$. A rigorous discussion of this question is beyond the scope of the present article, and we will leave it to future work.

\vspace{0.2cm}

\noindent
{\bf Algorithm 2: approximating theta-process by a hyper-exponential process.}

\noindent
Our second algorithm is based on approximating the theta-process $X$ by a hyper-exponential process $\tilde X=\tilde X^{(N)}$, for which the Mellin transform of the exponential functional can be computed explicitly. 
There is a natural way of obtaining such an approximation for any meromorphic process: We can simply truncate both infinite series defining the density of the  L\'evy measure in \eqref{def_pi} at $N$ terms (see \cite{Crosby} for another approximation technique). This procedure gives us a L\'evy process $\tilde X$ with hyper-exponential jumps, whose Laplace exponent is given by
 \beq\label{eq_tilde_psi}
 \tilde\psi(z)=\frac 12 \tilde \sigma^2 z^2 + \tilde\mu z
 +z^2 \sum\limits_{n=1}^N\frac{a_n }{\rho_n(\rho_n- z)}+z^2 \sum\limits_{n=1}^N \frac{\hat a_n}{\hat \rho_n (\hat \rho_n+z)},
 \eeq
see \eqref{eq_psi} for comparison. 
In our case, the coefficients $a_n$, $\rho_n$, $\hat a_n$ and $\hat \rho_n$ are defined by
equation \eqref{def_rho_a}. 
We will choose $\tilde \sigma$ so that the variance of $\tilde X_t$ matches the variance of $X_t$, which is equivalent to requiring $\tilde \psi''(0)=\psi''(0)$. The parameter $\tilde \mu$ is then specified by enforcing the risk-neutral condition $\tilde \psi(1)=r$. 

Following this simple procedure we obtain a sequence of hyper-exponential processes $\tilde X^{(N)}$, which will converge in Skorohod's topology to the process $X$. Now we can compute the price of the Asian option, with the driving process $\tilde X$, following the same procedure as in algorithm 1. The only difference is that the Mellin transform of the exponential functional $\tilde I_q=\int_0^{\ee(q)}\exp(\tilde X_t) \d t$ 
is computed as
\beqq
\tilde \mm(s,q)=\e \left[ \tilde I_q^{s-1} \right]=a\times \left(\frac{\tilde\sigma^2}{2} \right)^{1-s} \times \Gamma(s) \times 
\frac{\prod\limits_{j=1}^{N} \Gamma(\hat \rho_j+s)}{\prod\limits_{j=1}^{N+1} \Gamma(\hat\zeta_j+s) }
\times \frac{\prod\limits_{j=1}^{N+1} \Gamma(1+\zeta_j-s)}{\prod\limits_{j=1}^N \Gamma(1+\rho_j-s) },
\eeqq
where $a=a(q)$ is chosen so that $\tilde \mm(1,q)=1$. This expression 
for the Mellin transform was obtained in  \cite{CaiKou2010}. Here $\zeta_n$ and $-\hat \zeta_n$ are solutions to 
the equation $\tilde\psi(z)=q$, and since $\tilde \psi(z)$ is a rational function, it is easy to see that there will be exactly $N+1$ positive and $N+1$ negative solutions to $\psi(z)=q$ (see \cite{CaiKou2010} for all details). 

\vspace{0.2cm}

\noindent
{\bf Algorithm 3: Monte-Carlo simulation.}

\noindent
We will also check the accuracy of the previous two algorithms by computing the price of the Asian option by a simple Monte-Carlo simulation. 
We approximate the theta-process $X=\{X_t\}_{0\le t \le T}$ by a random walk $Z=\{Z_n\}_{0\le n\le 400}$ with $Z_0=0$ and the increment $Z_{n+1}-Z_n\stackrel{d}{=}X_{T/400}$. The price of the Asian option is approximated then by the following expectation
\beqq
e^{-r T} \e \left[\left( \frac{1}{400} \sum_{n=1}^{400} S_0e^{Z_n} - K\right)^+ \; \right],
\eeqq
which we estimate by sampling $10^6$ paths of the random walk. 
In order to sample from the distribution of $Y:=Z_{n+1}-Z_n$, we compute its density $p_Y(x)$ 
via the inverse Fourier transform 
\beqq
p_Y(x)=\frac{1}{2\pi} \int\limits_{\r} \e\left[e^{\i z Y}\right] e^{-\i z x} \d z,
\eeqq
where $\e\left[e^{\i z Y}\right]=\e\left[e^{\i z X_{T/400}}\right]=\exp\left((T/400)\psi(\i z)\right)$
and the Laplace exponent $\psi(z)$ is given by \eqref{def_theta}. Again, in order to compute the inverse Fourier transform, we use Filon's method with $10^6$ discretization 
points.

\vspace{0.2cm}

We compute the price of the Asian option with the initial stock price $S_0=100$, 
interest rate $r=0.03$, maturity $T=1$ and strike price $K=105$.  
We consider the two theta-processes defined by parameter sets I and II (see \eqref{def_parameter_sets}). 
Note that the parameter $\mu$ is not equal to $0.1$ anymore, as it has to be determined by the risk-neutral 
condition $\psi(1)=r$.  The results of our computations are presented in tables \ref{table1} and \ref{table2}. 
The code was written in Fortran90, and all computations were performed on a basic laptop 
(CPU: Intel Core i5-2540M 2.60GHz).

%
%
%
\begin{table} 
\centering
\begin{tabular}{| c | c | c | c |c|}
\hline
\rule{0pt}{3ex}   
$N$  & Algorithm 1, price  & Algorithm 1, time (sec.) & Algorithm 2, price & Algorithm 2, time (sec.) \\ [0.5ex]
\hline\hline
\rule{0pt}{2ex}  
10 & 4.724627 & 1.6 & 4.720675 & 1.2 
\\
\rule{0pt}{1ex}  
20 & 4.727780 & 2.8 & 4.728032 & 1.8 
\\
\rule{0pt}{1ex}  
40 & 4.728013 & 4.8 & 4.728031 & 3.4 
\\
\rule{0pt}{1ex}  
80 & 4.728029 & 9.2 & 4.728031 & 7.1 
\\
\hline
\end{tabular}
\caption{The price of the Asian option, parameter set I. The Monte-Carlo estimate of the price
is 4.7386 with the standard deviation 0.0172. The exact price is 4.72802$\pm$1.0e-5.}
\label{table1}
\centering
\vspace{0.3cm}
\begin{tabular}{| c | c | c | c |c|}
\hline
\rule{0pt}{3ex}   
$N$  & Algorithm 1, price  & Algorithm 1, time (sec.) & Algorithm 2, price & Algorithm 2, time (sec.)  \\ [0.5ex]
\hline\hline
\rule{0pt}{2ex}  
10 & 10.620243 & 1.6 & 10.621039 & 1.2 \\
\rule{0pt}{1ex}  
20 & 10.620049 & 3.0 & 10.620171 & 2.2 
\\
\rule{0pt}{1ex}  
40 & 10.620037 & 4.8 & 10.620054 & 3.6 
\\
\rule{0pt}{1ex}  
80 & 10.620036 & 9.6 & 10.620039 & 7.4 
\\
\hline
\end{tabular}
\caption{The price of the Asian option, parameter set II. The Monte-Carlo estimate of the price
is 10.6136 with the standard deviation 0.0251. The exact price is 10.62003$\pm$1.0e-5.}
\label{table2}
\end{table}

The results presented in tables \ref{table1} and \ref{table2} show that
both algorithm 1 and algorithm 2 perform very well, and seem to converge quickly to the true value of the option. 
The exact price was computed with $N=160$ and 1600 discretization points for the two integrals in
\eqref{hkq_inverse_Mellin},  \eqref{fkt_inverse_Laplace}. By experimenting with other values of the above parameters, we are convinced that these values are correct to withing 1.0e-5.  
Both of these algorithms are very efficient; the CPU time seems to be comparable with the results of Cai and Kou on hyper-exponential processes (see \cite{CaiKou2010}). 
Note that there is a substantial difference between algorithm 1 and algorithm 2, as one is based on 
approximating the Mellin transform of the exponential functional, and the other is based on
approximating the underlying L\'evy process by a hyper-exponential process, for which the Mellin transform
of the exponential functional can be computed explicitly. Yet the results produced by both of these algorithms agree 
up to five decimal digits, which is a good indicator that these results are indeed correct. 
The algorithm based on Monte-Carlo computation also produces consistent results, 
however these estimates are much less accurate and require CPU time on the order of several minutes. 

We have performed similar numerical experiments for many other values of parameters
of the underlying theta-process, as well as for different maturities and different strike prices. 
Qualitatively, the results seem to be consistent with the ones presented in tables \ref{table1} and \ref{table2}. 
Algorithm 2 is very efficient for the parameter set I, and gives high accuracy even for relatively small values of 
$N$. This should not be surprising, as in this case we have a theta-process with jumps of finite variation (but infinite activity), and it is intuitively clear that processes with compound Poisson jumps 
can provide a good approximation for such processes. On the other hand, parameter set II  corresponds to a theta-process with jumps of infinite variation and zero Gaussian component, and here our approximation by a compound Poisson process with a non-zero Gaussian component is bound to be less precise. 
Nevertheless, algorithm 2 works quite well in all cases, and, in our opinion, it may be a preferred
 method to compute prices of Asian options for any meromorphic processes (though testing this algorithm 
 on other meromorphic processes, such as beta-processes \cite{Kuz2010a}, would be worthwhile). 
 
 Algorithm 1 has comparable performance (though it is a little slower than algorithm 2). It has one potential advantage compared with algorithm 2. Imagine that 
 we want to compute the price of the Asian option for $N=20$ and then to check whether we have sufficient accuracy by doing the same computation with $N=40$. Algorithm 2 would require re-computing everything, since the Laplace exponent $\tilde \psi$ of the hyper-exponential process will change, thus all 
 the numbers $\zeta_n$, $-\hat \zeta_n$ (the solutions 
 to equation $\tilde \psi(z)=q$) will be different.  This is not the case for algorithm 1: 
here the Laplace exponent $\psi(z)$ does not depend on $N$, thus the roots $\zeta_n$ and $\hat \zeta_n$ 
with $1\le n \le 20$ will not be affected.  Therefore,  we only need to compute 
 the new roots for $21\le n \le 40$, which will need fewer computations. 
 The same idea can be applied for the evaluation of the Mellin transform, where the results for $N=20$ 
 in \eqref{def_Mn}
 can be stored in memory, and only the remaining finite product of gamma functions with $21\le n \le 40$ 
have to be evaluated.



\end{document}